\documentclass[10pt,conference]{ieeeconf}
\usepackage{amsmath,amssymb,latexsym,epsfig,psfrag,graphicx}
\usepackage{tikz,pgfplots}
\usetikzlibrary{positioning,calc}
\pgfplotsset{compat=1.16}
\newtheorem{prp}{Proposition}
\newtheorem{thm}[prp]{Theorem}

\newtheorem{lem}[prp]{Lemma}
\newenvironment{pf}{\noindent{\it Proof. }}{\hfill$\Box$\smallbreak}
\newenvironment{pf*}[1]{\smallbreak\noindent{\it #1}}{\hfill$\Box$\smallbreak}
\newcounter{definition}

\newcounter{remark}

\newcounter{example}

\newenvironment{ex*}[1]{\addtocounter{example}{1}\smallbreak\noindent
	{\bf Example \theexample{} --- {\bf #1}}}{\hfill$\Box$\smallbreak}
\newcommand{\realR}{\mathbb{R}}
\newcommand{\symmetricS}{\mathbb{S}}

\newcommand{\averageE}{\mathbb{E}}

\DeclareMathOperator{\conv}{conv}

\begin{document}
\title{Minimax optimal dual control --- The single input case}
  \author{Anders Rantzer
  \thanks{The author is affiliated with Automatic Control LTH, Lund
    University, Box 118, SE-221 00 Lund, Sweden. He is a member of the Excellence Center ELLIIT and Wallenberg AI, Autonomous Systems and Software Program (WASP). Support was received from the European Research Council (Advanced Grant 834142) }}
\maketitle


\medskip

\begin{abstract}
	An explicit solution is derived for the Bellman inequality corresponding to minimax optimal dual control. The minimizing player determines control action as a function of past state measurements and inputs. The maximizing player selects disturbances and model parameters for the underlying linear time-invariant dynamics. The optimal minimizing policy is a dual controller that optimizes the tradeoff between exploration and exploitation. Once sufficient data has been collected, the policy becomes a deterministic certainty equivalence controller. However, when data is insufficient, the policy introduces a randomized term to improve excitation.
\end{abstract}

\section{Introduction}
\label{sec:intro}

The term dual control was introduced by \cite{feldbaum1960dual} to describe the tradeoff between short term control objectives and actions to promote learning. 
This tradeoff is fundamentally important and has been studied extensively \cite{86helmersson+,88bernhardsson,wittenmark1995adaptive,filatov2000survey}. 
Dual control has recently received renewed attention in the context of machine learning. See \cite{flayac2021nonlinear} or the review papers \cite{mesbah2018stochastic,meijer2026dual}. 

In this paper, the focus is on worst-case models for disturbances and uncertain parameters, as discussed in \cite{cusumano1988nonlinear,sun1987theory,vinnicombe2004examples,2004megretskinonlinear} and more recently in \cite{rantzer2021minimax,cederberg2022synthesis,kjellqvist2022minimax,Rantzer2025acc}. 

The important difference of the current paper compared to previous literature, is that the explicit solution of the Bellman inequality from \cite{rantzer2025minimax} now is extended beyond sign uncertainty in the $B$-matrix and here allows for a norm bounded set of stabilizable systems.

The $A$-matrix is supposed to be known. The reason for this is mainly that it simplifies the formulas for solution of the Bellman inequality. Another reason is that the $B$-matrix is the  central system component of dual control, since unlike the $A$-matrix, $B$ cannot be studied without activation of the input.

\section{Notation and Preliminaries}
\label{sec:notation}
The set of $n\times m$ matrices with real coefficients is denoted $\realR^{n\times m}$. The transpose of a matrix $M$ is denoted $M^\top$. For a symmetric matrix $M\in\realR^{n\times n}$, we write $M\succ0$ to say that $M$ is positive definite, while $M\succeq0$ means positive semi-definite. The set of all $n\times n$ positive semi-definite matrices is denoted $\symmetricS_+^{n\times n}$. 

For $M, N\in\realR^{n\times m}$, the expression $\langle M,N\rangle$ denotes the trace of $M^\top N$. 
Given $x\in\realR^n$ and $M\in\realR^{n\times n}$, the notation $|x|^2_M$ means $x^\top Mx$. Similarly, given $N\in\realR^{n\times m}$ and $M\in\realR^{n\times n}$, the trace of $N^\top MN$ is denoted $\|N\|^2_M$. For two vectors $x\in\realR^n$, $y\in\realR^m$, the concatenation {\footnotesize$\begin{bmatrix}x\\y\end{bmatrix}$} will generally be denoted as $(x,y)$. 

\goodbreak

	For $A\in\realR^{n\times n}$, $S\in\symmetricS_+^{n\times n}$ and $\beta,R,\gamma\in\realR_+$, let $\mathcal{B}$ be the set of all $B\in\realR^n$ with $1\le |B|\le\beta$ and 
	{\small\begin{align}
		|x|^2\ge\min_u\left(|x|_S^2+R u^2+\frac{|Ax|^2+2B^\top Axu+\beta u^2}{1-\gamma^{-2}}\right).
	\label{eqn:Ricineq}
	\end{align}
	}
	The set is characterized by the following lemma, which follows from Lemma~\ref{lem:Bset} in the Appendix:
	\begin{lem}
		{\it Assume that the set $\mathcal{B}$ is non-empty but the matrix  $(1-\gamma^{-2})(I-S)-A^\top A$ has a negative eigenvalue. Then $\mathcal{B}=\mathcal{B}_-\cup\mathcal{B}_+$, where $\mathcal{B}_-=-\mathcal{B}_+$ and $\mathcal{B}_+$ is a convex second order cone. 
		}
	\label{lem:Bbar}
	\end{lem}
	Define also the set $\mathcal{\bar B}$ consisting of all of pairs $(b,B)$ such that $1\le b\le\beta$, $|B|^2\le b$ and \eqref{eqn:Ricineq} holds. 
	Furthermore, let 
	\begin{align*}
		z_{b,B}&:=\gamma^2\left\langle\hbox{\footnotesize$\begin{bmatrix}
			A^\top A&A^\top B&-A^\top\\
			B^\top A&b&-B^\top\\
			-A&-B&I
		\end{bmatrix},\underbrace{\begin{bmatrix}
		Z_{xx}&Z_{xu}&Z_{vv}\\
		Z_{ux}&Z_{uu}&Z_{uv}\\
		Z_{vx}&Z_{vu}&Z_{vv}
		\end{bmatrix}}_Z$}\right\rangle.
	\end{align*}

\begin{figure}
	\begin{center}
		\setlength{\unitlength}{.009mm}%
		\begin{picture}(9366,4600)(1318,-8461)
			\put(4501,-8461){\framebox(3000,1500){\begin{tabular}{ll}Causal\\ Controller $\mu$
			\end{tabular}}}
			\put(4000,-5811){\framebox(4200,1700){\begin{tabular}{l}
						$x_+=Ax+Bu+w$\\[2mm]
						$B\in\mathcal{B}$
			\end{tabular}}}
			\put(4000,-5161){\line(-1, 0){1150}}
			\put(2851,-5161){\line( 0,-1){2550}}
			\put(2851,-7711){\vector( 1, 0){1650}}
			\put(4000,-4861){\vector(-1, 0){2000}}
			\put(10001,-4861){\vector(-1, 0){1800}}
			\put(7501,-7711){\line( 1, 0){1650}}
			\put(9151,-7711){\line( 0, 1){2550}}
			\put(9151,-5161){\vector(-1, 0){950}}
			\put(2300,-6500){$x$}%
			\put(9500,-6500){$u$}%
			\put(2200,-4500){$(x,u)$}%
			\put(9500,-4500){$w$}%
		\end{picture}%
	\end{center}
	\caption{We want a feedback controller that works for all $B$-vectors within given bounds. If $A$ is unstable and $\mathcal{B}$ includes sign uncertainty, even stabilization is impossible using linear time-invariant controllers. A randomized nonlinear dual controller can do much better by estimating $B$ and use the estimate for control. A guaranteed bound is derived for the nonlinear gain from $|w|^2$ to the expected value of $|x|^2_S+|u|^2_R$.
	}
	\label{fig:blockdiag}
\end{figure}

\section{Main results}

\begin{thm}
	{\it Given notation of section~\ref{sec:notation}, suppose that 
	{\begin{align}
		\gamma^2&\ge
		\left(1+\frac{4\|A\|^2}{1-\gamma^{-2}}\right)\max\{1+R+\beta^2,\|S^{-1}\|\}.
		\label{eqn:tau-gamma}
	\end{align} 	
	}and $\mathcal{B}$ is non-empty. Then there exists a control law $u=\mu(x,Z)$, generating (possibly random) inputs $u_t$, such that
	{\begin{align}
		\averageE\sum_{t=0}^{N}\left(|x_t|_S^2+|u_t|_R^2\right)
		\le\gamma^2\sum_{t=0}^{N}|w_t|^2
	\label{eqn:gain}
	\end{align}
	}whenever $B\in\mathcal{B}$ and 
	\begin{align}
		x_{t+1}&=Ax_t+Bu_t+w_t, \quad x_0=0,\label{eqn:system}\\
		u_t&=\mu\left(x_t,\sum_{k=0}^{t-1}(x_k,u_k,x_{k+1})(x_k,u_k,x_{k+1})^\top\right).\label{eqn:controllaw}
	\end{align}
	In particular, the inequality \eqref{eqn:gain} holds for the feedback law $\mu(x,Z)$, defined as the minimizing argument of
	{\small\begin{align}
		\min_{u\in\mathcal{U}_x}\max_{(b,B)\in\mathcal{\bar B}}
		\averageE\left[R u^2+\frac{|Ax|^2+2B^\top Axu+b u^2}{1-\gamma^{-2}}-z_{b,B}\right].\label{eqn:Bhat}
	\end{align}
	}where $\mathcal{U}_x$ is the set of all random $u\in\realR$ with
	\begin{align}
		\averageE u^2\ge\frac{|Ax|^2-(1-\gamma^{-2})|x|^2_{I-S}}{\beta+(1-\gamma^{-2})R}.
	\label{eqn:ubound}
	\end{align}}
	\label{thm:Malpha}
\end{thm}
The proof of Theorem~\ref{thm:Malpha} will be based on the following main technical result:

\medskip

\begin{thm}[Solution to the Bellman inequality]
	{\it With the assumptions of 
		Theorem~\ref{thm:Malpha}, 
		define
		\begin{align*}
			\tau&:=1+4\|A\|^2(1-\gamma^{-2})^{-1}\\
			\bar z&:=\min_{B_+\in\mathcal{B}_+}\min_{B_-\in\mathcal{B}_-}(z_{1,B_+}+z_{1,B_-})/2.
		\end{align*}
		Then 
		\begingroup\small
		\begin{align*}
			\hat{V}(x,Z)
			&:=\max\left\{\max_{(b,B)\in\mathcal{\bar B}}\averageE(|x|^2-z_{b,B}),\averageE\left(\tau|x|^2-\bar z\right)\right\}.
		\end{align*}
		\endgroup
		satisfies
		\begingroup\small
		\begin{align}
			\hat{V}(x,Z)
			&\ge\max_v\left[|x|^2_S+R\averageE\hat{u}^2+\hat{V}\left(v,Z+\hbox{\footnotesize$
			\begin{bmatrix}x\\\hat{u}\\v\end{bmatrix}
			\begin{bmatrix}x\\\hat{u}\\v\end{bmatrix}^\top
			$}\right)\right],\label{eqn:Bellman}
		\end{align}
		\endgroup
		where $\hat{u}=\mu(x,Z)$ and $v$ is random and $\hat{u}$-dependent.}
	\label{thm:Bellman}
\end{thm}

\medskip

\begin{pf}
	Given a random $u\in\realR$, define the operator $\mathcal{F}_u$ by
	{\small\begin{align*}
			\mathcal{F}_uV(x,Z)
			&:=
			\max_v\left[|x|^2_S+R\averageE u^2+V\left(v,Z+\hbox{\footnotesize$
				\begin{bmatrix}x\\u\\v\end{bmatrix}
				\begin{bmatrix}x\\u\\v\end{bmatrix}^\top
				$}\right)\right]
		\end{align*}
	}where $v$ is a random number that may depend on $u$. Let
	{\small\begin{align}
		V_{b,B}(x,Z)&:=\averageE(|x|^2-z_{b,B})\notag\\
		\bar{V}(x,Z)&:=\averageE\left(\tau|x|^2-\bar z\right).
		\label{eqn:Vbar}
	\end{align}
	}and note that
	{\small\begin{align*}
		&\mathcal{F}_uV_{b,B}(x,Z)\\
		&=\averageE\left(|x|_S^2+R u^2+\frac{|Ax|^2+2(Ax)^\top Bu+bu^2}{1-\gamma^{-2}}-z_{b,B}\right).
	\end{align*}
	}The desired inequality \eqref{eqn:Bellman} can be written as $\mathcal{F}_{\hat{u}}\hat{V}\le\hat{V}$. Lemma~\ref{lem:Vbarbound} in the Appendix gives $\mathcal{F}_{\hat{u}}\bar{V}\le \bar{V}$, so it remains to prove that $\mathcal{F}_{\hat{u}}V_{b,B}(x,Z)\le\hat{V}(x,Z)$ for all $(b,B)\in\mathcal{\bar B}$. 
	
\goodbreak
	Given the fact that $\mathcal{F}_{u}V_{b,B}(x,Z)$ is linear in $(\averageE u,\averageE u^2)$ and linear in $(b,B)$, it follows from von Neumann's minimax theorem that
\begingroup
\begin{align}
	\min_{u\in\mathcal{U}_x}\max_{(b,B)\in\mathcal{\bar B}}\mathcal{F}_uV_{b,B}(x,Z)
	&=\max_{(b,B)\in\conv\mathcal{\bar B}}\min_{u\in\mathcal{U}_x}\mathcal{F}_uV_{b,B}(x,Z).\label{eqn:minimax}
\end{align}
\endgroup
Let $(\hat b,\hat B)$ be a maximizing argument of the right hand side.
By the definition in Theorem~\ref{thm:Malpha}, $\hat u$ is a minimizing argument of the left hand side in \eqref{eqn:minimax}. 
Hence, $\hat u$ together with $(\hat b,\hat B)$ defines a saddle point with the value 
{\small\begin{align}
		&\mathcal{F}_{\hat{u}}V_{\hat b,\hat B}(x,Z)\notag\\
		&=\averageE\left[|x|_S^2+R\hat u^2+\frac{|Ax|^2+2\hat B^\top Ax\hat u+\hat b\hat u^2}{1-\gamma^{-2}}-z_{\hat b,\hat B}\right].\label{eqn:hathat}
	\end{align}
}It remains to prove that the saddle point value is bounded above by $\hat{V}(x,Z)$.
	Two cases need to be considered:

\medskip\goodbreak

\noindent\textbf{Case~1 (Exploitation):} $(\hat b,\hat B)\in\mathcal{\bar B}$
\hfil\smallbreak
\noindent
In this case
{\begin{align*}
	&\mathcal{F}_{\hat{u}}V_{\hat b,\hat B}(x,Z)\\
	&=\min_u\left[|x|_S^2+R u^2+\frac{|Ax|^2+2\hat B^\top Ax u+\hat b u^2}{1-\gamma^{-2}}-z_{\hat b,\hat B}\right]\\
	&\le|x|^2-z_{\hat{b},\hat{B}}\\
	&\le\hat{V}(x,Z)
\end{align*}
}according to \eqref{eqn:Ricineq}. Hence \eqref{eqn:Bellman} holds.

\medskip\goodbreak

\noindent\textbf{Case~2 (Exploration):} $(\hat b,\hat B)\not\in\mathcal{\bar B}$.
\hfil\smallbreak
\noindent	
In this case $\hat B$ must have the form
\begin{align*}
	 \hat B=\hat\theta B_++(1-\hat\theta)B_-
\end{align*}
where $\hat\theta\in(0,1)$, $B_+\in\mathcal{B}_+$ and $B_-\in\mathcal{B}_-$.
Introduce the notation $B_\theta:=\theta B_++(1-\theta)B_-$ for $\theta\in[0,1]$. In order for $\hat\theta$ to be maximizing, the expression
{\begin{align*}
	B_\theta^\top\left[\frac{Ax\averageE\hat u}{1-\gamma^{-2}}-AZ_{xu}+Z_{vu}\right]
\end{align*}
}must have the same value for all $\theta\in[0,1]$. In other words,
\begin{align*}
	\averageE\left[\frac{2(Ax)^\top B_\theta\hat u}{1-\gamma^{-2}}-z_{\hat b,B_\theta}\right]
\end{align*}
has the same value for all $\theta\in[0,1]$. Hence
{\small\begin{align*}
	&\mathcal{F}_{\hat{u}}V_{\hat b,\hat B}(x,Z)-\averageE\left(|x|_S^2+R\hat u^2\right)\\
	&=\averageE\left[\frac{|Ax|^2+2(Ax)^\top\hat B\hat u+\hat b \hat u^2}{1-\gamma^{-2}}-z_{\hat b,\hat B}\right]\\
	&=\averageE\left[\frac{|Ax|^2
	+(Ax)^\top(B_++B_-)\hat u+\hat b \hat u^2}{1-\gamma^{-2}}-\frac{z_{\hat b,B_+}+z_{\hat b,B_-}}{2}\right]\\
	&\le\averageE\left[\frac{4|Ax|^2}{1-\gamma^{-2}}-\bar z\right]\\
	&=\bar{V}(x,Z),
\end{align*}
}so \eqref{eqn:Bellman} follows also in this case.

\end{pf}

\goodbreak

\begin{pf*}{Proof of Theorem~\ref{thm:Malpha}.}
	If $(1-\gamma^{-2})(I-S)\succeq A^\top A$, then the control law $u_t\equiv0$ is sufficient. If not, then
	consider a solution to \eqref{eqn:system}-\eqref{eqn:controllaw} with $\mu$ defined by \eqref{eqn:Bhat}-\eqref{eqn:ubound}.
	Let $Z_0:=0$ and 
	\begin{align*}
		Z_t&:=\sum_{k=0}^{t-1}(x_k,u_k,x_{k+1})(x_k,u_k,x_{k+1})^\top
		&&\text{for }t=1,\ldots T.
	\end{align*}
	Then
	{\begin{align*}
		&\hat{V}(x_{t+1},Z_{t+1})+\averageE|u_t|_R^2\\
		&=\hat{V}(x_{t+1},Z_t+(x_t,u_t,x_{t+1})(x_t,u_t,x_{t+1})^\top)+\averageE|u_t|_R^2\\
		&\le\max_v\hat{V}(v,Z_t+(x_t,u_t,v)(x_t,u_t,v)^\top)+\averageE|u_t|_R^2\\
		&=\max_v\hat{V}(v,Z_t+(x_t,\mu(x_t,Z_t),v)(x_t,\mu(x_t,Z_t),v)^\top)\\
		&\quad +\averageE|\mu(x_t,Z_t)|_R^2\\
		&\le\left[\hat{V}(x_t,Z_t)-\averageE|x_t|_S^2\right],
	\end{align*}
	}where the last inequality is due to Theorem~\ref{thm:Bellman}. Hence
	\begin{align*}
		&\averageE\sum_{t=0}^{T}\left(|x_t|_S^2+|u_t|_R^2\right)\\
		&\le\sum_{t=0}^{T}\left[\hat{V}(x_t,Z_t)-\hat{V}(x_{t+1},Z_{t+1})\right]\\
		&=\hat{V}(x_0,Z_0)-\hat{V}(x_{T+1},Z_{T+1})\\
		&=-\hat{V}(x_{T+1},Z_{T+1})\\
		&\le\gamma^2\begin{bmatrix}A&B&-I\end{bmatrix}\left(\averageE Z_T\right)
		\begin{bmatrix}A&B&-I\end{bmatrix}^\top\\
		&=\gamma^2\averageE\sum_{t=0}^{T}|w_t|^2
	\end{align*}
	for all $B\in\mathcal{B}$ and the proof is complete.
\end{pf*}

\section{Example}
To illustrate the results above, consider the unstable state space model
\begin{align}
	x_{t+1}&=\begin{bmatrix}
		0.98&-0.24\\
		-0.24&0.62\end{bmatrix}x_t+
	\begin{bmatrix}
		b_1\\
		b_2\end{bmatrix}u_t+w_t.
	\label{eqn:IOstate}
\end{align}
with open loop eigenvalues $1.1$ and $0.5$. Let $\beta=4$, $\gamma=10$, $S=0.1I$ and $R=0.01$. Then Theorem~\ref{thm:Malpha} can be applied to prove that the minimax dual controller \eqref{eqn:controllaw} stabilizes the system \eqref{eqn:IOstate} with uncertainty set $\mathcal{B}$ and gain bound $\gamma=10$. Simulation of the system for $(b_1,b_2)=(1.2,-0.6)$ with the controller \eqref{eqn:controllaw} and $w_t$ being white noise with unit variance gives the plot of Figure~\ref{fig:example}.

\begin{figure}

\definecolor{mycolor1}{rgb}{0.92900,0.69400,0.12500}
\definecolor{mycolor2}{rgb}{0.12941,0.12941,0.12941}%

\ifdefined\tikzplotheight\else\newlength{\tikzplotheight}\fi
\ifdefined\vplotspacing\else\newlength{\vplotspacing}\fi
\ifdefined\tikzplotwidth\else\newlength{\tikzplotwidth}\fi
\setlength{\tikzplotheight}{3cm}
\setlength{\vplotspacing}{.45cm}
\setlength{\tikzplotwidth}{.75\linewidth}

\def\xMinAll{0}
\def\xMaxAll{60}
\def\yMinNorm{0}
\def\yMaxNorm{30}
\def\yMinB{-1}
\def\yMaxB{4}

\begin{tikzpicture}

\pgfplotsset{
  vaxisbase/.style={
    width=\tikzplotwidth,
    height=\tikzplotheight,
    scale only axis,
    xmin=\xMinAll, xmax=\xMaxAll,
    xtick={0,20,40,60},
    axis background/.style={fill=none},
    clip=true,
  },
}

\begin{axis}[
  vaxisbase,
  at={(0,0)},
  axis x line*=bottom,
  axis y line*=left,
  xlabel style={font=\footnotesize\color{mycolor2}},
  xlabel={$t$},
  ymin=\yMinNorm, ymax=\yMaxNorm,
  ylabel={$\|x(t)\|,~|u(t)|$},
  ylabel style={font=\footnotesize},
  every axis title shift=0pt,
  title={Minimax approach},
  title style={at={(0.5,1)}, anchor=south, font=\scriptsize, inner sep=1pt, yshift=1pt},
  ymajorgrids, xmajorgrids,
]
\fill[mycolor1, fill opacity=0.35, draw=none]
  (axis cs:3,\yMinNorm) rectangle (axis cs:4,\yMaxNorm);
\addplot [color=mycolor1, line width=1.5pt, forget plot] table[row sep=crcr]{%
0	25\\
1	14.9005\\
2	7.3388\\
3	4.46418\\
4	4.00953\\
5	8.7679\\
6	4.61186\\
7	2.80852\\
8	1.90616\\
9	1.98567\\
10	2.95378\\
11	1.65074\\
12	2.16885\\
13	2.21787\\
14	1.54592\\
15	1.0737\\
16	1.82184\\
17	1.51542\\
18	3.4275\\
19	1.10799\\
20	2.2209\\
21	1.11359\\
22	0.309779\\
23	0.791922\\
24	0.723603\\
25	1.63652\\
26	0.823834\\
27	1.51875\\
28	1.24029\\
29	1.20867\\
30	1.50015\\
31	1.26644\\
32	1.31683\\
33	1.89971\\
34	0.883736\\
35	2.54736\\
36	1.7948\\
37	1.5249\\
38	1.01069\\
39	1.69263\\
40	0.711551\\
41	1.21691\\
42	1.83466\\
43	1.58339\\
44	0.990264\\
45	1.03561\\
46	1.15047\\
47	2.38792\\
48	1.31281\\
49	1.40529\\
50	2.48954\\
51	2.0665\\
52	1.38743\\
53	1.22211\\
54	0.433395\\
55	1.82827\\
56	1.82628\\
57	1.78569\\
58	1.14019\\
59	2.62717\\
60	1.89051\\
61	2.68691\\
};
\addplot [color=mycolor1, dashed, line width=1.5pt, const plot, forget plot] table[row sep=crcr]{%
0	0\\
1	0\\
2	0\\
3	0.140213\\
4	0.932199\\
5	3.31578\\
6	0.466918\\
7	0.610777\\
8	0.665381\\
9	0.20264\\
10	0.24789\\
11	0.800286\\
12	0.177865\\
13	0.0546792\\
14	0.502187\\
15	0.579891\\
16	0.958162\\
17	0.0457741\\
18	1.81046\\
19	0.31466\\
20	1.00802\\
21	0.355248\\
22	0.174269\\
23	0.357733\\
24	0.389913\\
25	0.109405\\
26	0.388037\\
27	0.41544\\
28	0.601747\\
29	0.69011\\
30	0.654562\\
31	0.77743\\
32	0.237168\\
33	0.765586\\
34	0.37253\\
35	1.36477\\
36	0.949913\\
37	0.591054\\
38	0.163403\\
39	1.03447\\
40	0.2564\\
41	0.711248\\
42	0.233065\\
43	0.731361\\
44	0.432041\\
45	0.580742\\
46	0.598418\\
47	0.4599\\
48	0.560821\\
49	0.41756\\
50	0.350293\\
51	0.851292\\
52	0.390431\\
53	0.444023\\
54	0.1632\\
55	1.15284\\
56	0.870287\\
57	0.955408\\
58	0.765629\\
59	1.63758\\
60	0.515218\\
61	1.76462\\
};
\end{axis}

\begin{axis}[
  vaxisbase,
  at={(0,0)},
  axis x line=none,
  axis y line*=right,
  ymin=\yMinB, ymax=\yMaxB,
  ylabel={$B,~\hat{B}_t$},
  ylabel style={color=black, font=\footnotesize},
  legend style={legend cell align=left, font=\footnotesize,
    at={(0.98,0.94)}, anchor=north east, inner sep=1pt, draw=none,
    fill=white, fill opacity=0.8, text opacity=1,
    nodes={inner sep=1pt},
    /tikz/every even column/.append style={column sep=6pt}},
  legend image post style={xscale=0.5},
  legend columns=-1,
]
\addlegendimage{color=mycolor1, line width=1.5pt}
\addlegendentry{$\|x(t)\|$}
\addlegendimage{color=mycolor1, dashed, line width=1.5pt}
\addlegendentry{$|u(t)|$}
\addlegendimage{color=mycolor1, dotted, line width=1.5pt}
\addlegendentry{$\hat{B}_t$}
\addlegendimage{color=black!50, line width=1.5pt}
\addlegendentry{$B$}
\addplot [color=black!50, line width=1.5pt, forget plot] table[row sep=crcr]{%
0	1.2\\
60	1.2\\
};
\addplot [color=black!50, line width=1.5pt, forget plot] table[row sep=crcr]{%
0	-0.6\\
60	-0.6\\
};
\addplot [color=mycolor1, dotted, line width=1.5pt, forget plot] table[row sep=crcr]{%
0	1.13018\\
1	1.13018\\
2	1.13018\\
3	1.43504\\
4	-0.301556\\
5	1.76827\\
6	1.66749\\
7	1.66225\\
8	1.65918\\
9	1.63072\\
10	1.64731\\
11	1.65398\\
12	1.602\\
13	1.59404\\
14	1.59347\\
15	1.60563\\
16	1.60705\\
17	1.64985\\
18	1.65809\\
19	1.6091\\
20	1.63031\\
21	1.61317\\
22	1.60474\\
23	1.61172\\
24	1.60614\\
25	1.62431\\
26	1.62881\\
27	1.60218\\
28	1.59876\\
29	1.56942\\
30	1.55384\\
31	1.5138\\
32	1.52464\\
33	1.54007\\
34	1.58138\\
35	1.56209\\
36	1.50116\\
37	1.48351\\
38	1.46234\\
39	1.45314\\
40	1.4365\\
41	1.44924\\
42	1.43972\\
43	1.43255\\
44	1.44959\\
45	1.44475\\
46	1.46489\\
47	1.47653\\
48	1.46411\\
49	1.45976\\
50	1.46439\\
51	1.45357\\
52	1.44478\\
53	1.43975\\
54	1.43122\\
55	1.42201\\
56	1.48872\\
57	1.5177\\
58	1.48231\\
59	1.44605\\
60	1.43701\\
61	1.3955\\
};
\addplot [color=mycolor1, dotted, line width=1.5pt, forget plot] table[row sep=crcr]{%
0	-0.202776\\
1	-0.202776\\
2	-0.202776\\
3	-0.609832\\
4	1.97713\\
5	0.934456\\
6	-1.10431\\
7	-1.11217\\
8	-1.11675\\
9	-1.1579\\
10	-1.13419\\
11	-1.12443\\
12	-1.19734\\
13	-1.20791\\
14	-1.20865\\
15	-1.19246\\
16	-1.16676\\
17	-1.08149\\
18	-1.08659\\
19	-1.07365\\
20	-1.09984\\
21	-1.06163\\
22	-1.06437\\
23	-1.06344\\
24	-1.0491\\
25	-1.02415\\
26	-1.02425\\
27	-1.04264\\
28	-1.00744\\
29	-0.98959\\
30	-0.944374\\
31	-0.93671\\
32	-0.889002\\
33	-0.88765\\
34	-0.855265\\
35	-0.815312\\
36	-0.758288\\
37	-0.720055\\
38	-0.750321\\
39	-0.747219\\
40	-0.748784\\
41	-0.745733\\
42	-0.78356\\
43	-0.777312\\
44	-0.780429\\
45	-0.761524\\
46	-0.761863\\
47	-0.706798\\
48	-0.703184\\
49	-0.738484\\
50	-0.712626\\
51	-0.701744\\
52	-0.684448\\
53	-0.67647\\
54	-0.686596\\
55	-0.685619\\
56	-0.652394\\
57	-0.67894\\
58	-0.680442\\
59	-0.635633\\
60	-0.580444\\
61	-0.592151\\
};
\end{axis}

\end{tikzpicture}%
\caption{Simulation of the second order unstable system \eqref{eqn:IOstate} initialized closed to the stable eigenvector. The controller is not active during the first three time steps, since exploration is not worthwhile in this region of the state space. Hence no information about $B$ is collected in this time interval and $\hat{B}_t$ is purely representing computation of worst case parameters values. 
Finally, at $t=4$, the shaded staple indicates that the control signal is randomized for one time step and the parameter vector is quickly learnt. }
\label{fig:example}
\end{figure}

\section{Conclusions and Future Work}
The results of this paper demonstrate the power of the minimax paradigm for dual control. However, much works remains to be done. The most obvious next steps are the following three: 1) Extend to joint estimation of $A$ and $B$. 2) Extend to multiple inputs. 3) Extend to systems that are not stabilizable with the Lyapunov function $|x|^2$.

\section{Acknowledgements}
The author is a member of the Excellence Center ELLIIT and Wallenberg AI, Autonomous Systems and Software Program (WASP). Support was received from the European Research Council (DualControl, Advanced Grant 101199738).


\section{Appendix}


\medskip

\begin{lem}
	{\it Given $R,S,\gamma$ and $A$, suppose that the matrix $M=A^\top A-(1-\gamma^{-2})(I-S)$ has eigenvalues $\lambda_1\ge\ldots\ge\lambda_n$ with orthonormal eigenvectors $U_1,\ldots,U_n$ and $\lambda_1>0$. Consider $\mathcal{B}$ of Definition~1. Then  
		$\mathcal{B}=\realR^n$ if $\lambda_1\le0$ and $\mathcal{B}=\emptyset$ if $\lambda_2>0$. If $\lambda_2<0$, then $\mathcal{B}=\mathcal{B}_-\cup\mathcal{B}_+$, where $\mathcal{B}_-=-\mathcal{B}_+$ and $\mathcal{B}_+$ is the set of all $B$ such that $1\le |B|^2\le\beta$ and
		\begingroup\small
		\begin{align*}
			B^\top AU_1\ge\sqrt{\lambda_1\left(\beta+(1-\gamma^{-2})R-\sum_{k=2}^n\frac{(B^\top AU_k)^2}{\lambda_k}\right)}
		\end{align*}
		\endgroup	
	}
	\label{lem:Bset}
\end{lem}

\begin{pf*}{Proof of Lemma~\ref{lem:Bset}.}
	By definition, $\mathcal{B}_+$ is the set of $B\in\realR^n$ such that 
	\begin{align*}
		|x|^2&\ge \min_u\left(|x|_S^2+R u^2+\frac{|Ax|^2+2B^\top Axu+\beta u^2}{1-\gamma^{-2}}\right)\\
		&=|x|^2_S+\frac{|Ax|^2-(\beta+(1-\gamma^{-2})R)^{-1}(B^\top Ax)^2}{1-\gamma^{-2}}
	\end{align*}
	for all $x$. Equivalently
	\begin{align}
		A^\top  B\left[\beta+(1-\gamma^{-2})R\right]^{-1}B^\top A\succeq M.
		\label{eqn:1}
	\end{align}
	The left hand side of \eqref{eqn:1} has rank one, so $\mathcal{B}_+$ is empty unless $\lambda_2\le0$. Conversely, if $\lambda_2\le0$, then \eqref{eqn:1} can be rewritten as
	\begin{align*}
		\beta+(1-\gamma^{-2})R\preceq B^\top A M^{-1} A^\top  B&\\
		\lambda_1\left[\beta+(1-\gamma^{-2})R-\sum_{k=2}^n\lambda_k^{-1}(B^\top AU_k)^2\right]
		&\preceq (B^\top AU_1)^2,
	\end{align*}
	which means that either $B^\top AU_1$ or $-B^\top AU_1$ is bounded below by the square root of the left hand side. This proves that $\mathcal{B}=\mathcal{B}_-\cup\mathcal{B}_+$.
\end{pf*}


\begin{lem}
	The assumptions of Theorem~\ref{thm:Malpha}-\ref{thm:Bellman} imply that $\bar V$ defined by \eqref{eqn:Vbar} satisfies
	\begin{align*}
		\mathcal{F}_{\hat u}\bar{V}(x,Z)\le \bar{V}(x,Z).
	\end{align*}
	for all $x,Z$.
	\label{lem:Vbarbound}
\end{lem}

\begin{pf}
	First note that 
	{\small\begin{align*}
			&\mathcal{F}_{\hat u}\bar{V}(x,0)-\averageE\left(|x|_S^2+R\hat u^2\right)\\
			&=\max_{B_+\in\mathcal{B}_+}\max_{B_-\in\mathcal{B}_-}\max_v\averageE\left(\tau|v|^2-\frac{\gamma^2}{2}\sum_{i=\pm}|Ax+B_i\hat u-v|^2\right)\\
			&=\max_{B\in\mathcal{B}_+}\max_v\averageE\left(\tau|v|^2-\frac{\gamma^2}{2}\sum_{j=\pm 1}|Ax+jB\hat u-v|^2\right)\\
			&=\max_{B\in\mathcal{B}_+}\max_v\averageE\left(\tau|v|^2-\gamma^2|Ax-v|^2-\gamma^2|B\hat u|^2\right)\\
			&=\max_{B\in\mathcal{B}_+}\averageE\left(\frac{|Ax|^2}{\tau^{-1}-\gamma^{-2}}-\gamma^2|B\hat u|^2\right)\\
			&=\averageE\left(\frac{|Ax|^2}{\tau^{-1}-\gamma^{-2}}-\gamma^2\hat u^2\right),
		\end{align*}
	}where the second equation holds by symmetry between $B_+$ and $B_-$.
	Hence
	{\small\begin{align*}
		\mathcal{F}_{\hat u}\bar{V}(x,Z)
		&\le\averageE\left(|x|_S^2+R\hat u^2+\frac{|Ax|^2}{\tau^{-1}-\gamma^{-2}}-\gamma^2\hat u^2-\bar z\right).
	\end{align*}
	}To prove that $\mathcal{F}_{\hat u}\bar{V}(x,Z)\le\bar V(x,Z)$, it therefore remains to verify the inequality
	{\small\begin{align}
		\tau|x|^2&\succeq |x|^2_S- (\gamma^2-R)\averageE\hat{u}^2+(\tau^{-1}-\gamma^{-2})^{-1}|Ax|^2.
	\label{eqn:gamma}
	\end{align}
	}or equivalently
	{\small\begin{align}
		\left(\tau^{-1}-\gamma^{-2}\right)|x|_{(\tau I-S)}^2
		&\ge|Ax|^2-\left(\tau^{-1}-\gamma^{-2}\right)(\gamma^2-R)\averageE\hat{u}^2
	\label{eqn:desired}
	\end{align}
	}	
	The condition $\hat{u}\in\mathcal{U}_x$ gives
	\begin{align}
		(1-\gamma^{-2})|x|_{(I-S)}^2
		&\ge |Ax|^2-((1-\gamma^{-2})R+\beta)\averageE\hat{u}^2.
	\label{eqn:known}
	\end{align}
	so it remains to prove that the desired inequality \eqref{eqn:desired} follows from the known inequality \eqref{eqn:known}.
	Comparing the left hand sides, we have 
	\begin{align*}
		&(\tau^{-1}-\gamma^{-2})(\tau I-S)-(1-\gamma^{-2})(I-S)\\
		&=\gamma^{-2}(1-\tau^{-1})(\gamma^2S-\tau I)\succeq0,
	\end{align*}
	since $\tau\ge1$ and $\gamma^2S\succeq\tau I$. Hence the left side of \eqref{eqn:desired} is bounded below by the left side of \eqref{eqn:known}. For the right hand sides, it remains to verify that
	\begin{align*}
		(\tau^{-1}-\gamma^{-2})(\gamma^2-R)
		&\ge(1-\gamma^{-2})R+\beta
	\end{align*}
	or the slightly stronger condition 
	\begin{align*}
		\gamma^2\tau^{-1}&\ge1+2R+\beta.
	\end{align*}
	This follows from \eqref{eqn:tau-gamma}, so the proof is complete.
\end{pf}

\newpage

\end{document}